\documentclass[12pt, a4paper]{amsart}

\usepackage{latexsym, graphicx}

\usepackage{amssymb}
\usepackage[centertags]{amsmath}
\usepackage{verbatim}
\usepackage{epic, eepic}
\usepackage{enumerate}
\usepackage{color}

\newtheorem{theorem}{Theorem}
\newtheorem{lemma}[theorem]{Lemma}
\newtheorem{proposition}[theorem]{Proposition}
\newtheorem{corollary}[theorem]{Corollary}
\newtheorem{definition}[theorem]{Definition}
\newtheorem{remark}[theorem]{Remark}

\numberwithin{theorem}{section}

\renewcommand{\phi}{\varphi}

\renewcommand{\(}{\bigl(}

\newcommand{\const}{\operatorname{const}}

\newcommand{\M}{\mathcal M}

\renewcommand{\epsilon}{\varepsilon}

\begin{document}

\title[Magnetic billiards on Sphere and Hyperbolic plane]
      {Algebraic non-integrability of magnetic billiards on the Sphere and Hyperbolic plane}

\date{27 January 2018}
\author{Misha Bialy and Andrey E. Mironov}
\address{M. Bialy, School of Mathematical Sciences,
	Raymond and Beverly Sackler Faculty of Exact Sciences, Tel Aviv
	University, Israel.} \email{bialy@post.tau.ac.il}
\address{A.E. Mironov, Sobolev Institute of Mathematics,
4 Acad. Koptyug avenue, 630090, Novosibirsk, Russia and
Novosibirsk State University,
 Pirogova st 1, 630090, Novosibirsk, Russia}
\email{mironov@math.nsc.ru}
\thanks{M.B. was supported in part by ISF grant 162/15 and A.E.M. was
	supported by RSF (grant 14-11-00441). It is our pleasure to thank
	these funds for the support}

\subjclass[2000]{ } 
\maketitle
\begin{center}
	\it To the 80th birthday of Sergey Petrovich Novikov with great respect
\end{center}

\begin{abstract}We consider billiard ball motion in
a convex domain on a constant curvature surface influenced by the
constant magnetic field. We examine the existence of integral of motion which is polynomial in velocities. We prove that if such an integral exists then the boundary curve of the domain determines an algebraic curve in $\mathbf{C}^3$ which must be nonsingular.
Using this fact we deduce that for any domain different from round disc for all but finitely many values of the magnitude of the magnetic field billiard motion does not have Polynomial in velocities integral of motion.
\end{abstract}



\section{\bf Introduction }In this paper we consider
a magnetic billiard inside a convex domain $\Omega\subset\Sigma$ of the surface $\Sigma$ of constant curvature $\pm 1$.
 The domain is assumed to be bounded by a simple  smooth
closed curve $\gamma$. We consider the influence of a magnetic field of
constant magnitude $\beta>0$ on the billiard motion, so that the
particle moves inside $\Omega$ with unit speed along a Larmor circle
of constant geodesic curvature $\beta$ and geodesic radius $r$ where $\beta$ and $r$ are related as follows.
In the spherical case $\beta=\cot r$ , while in the case of Hyperbolic plane the condition that the trajectories of the magnetic flow are circles means precisely
that $\beta>1$ and $\beta=\coth r$. It is important to mention that Larmor circles come with the orientation so that the disc they are bounding lies to the left.

Upon hitting the boundary of $\Omega$, the billiard particle is reflected
according to the law of geometric optics. We call such a model a
magnetic Birkhoff billiard.

Throughout the paper we shall assume that the boundary $\gamma$ of
$\Omega$ satisfies
$$
\beta< \min_{\gamma} k,
$$
where $k$ is the curvature. Under this condition the  billiard ball dynamics is correctly defined for all times. Notice that in the Hyperbolic case the condition in particular means that $\gamma$ is convex with respect to horocycles.

\begin{remark}\label{key}
	It is plausible that the results below can be generalized to other ranges of the magnitude $\beta$,
	but we couldn't verify this by our methods. It is especially interesting to treat the case of billiards on the Hyperbolic plane with $0<\beta\leq 1$.
\end{remark}

Billiards is a very rich and interesting subject (see the books \cite{T}, \cite{kt}). Magnetic
Birkhoff billiards were studied in many papers; see, e.g.,
\cite{Berg}, \cite{B}, \cite{BR}, \cite{GB},
 \cite{T2}. The question of existence of Polynomial integrals is very natural and surprisingly deep (see for example the survey \cite{ko}). In our recent paper \cite{BM3} we studied  the question of polynomial integrability of magnetic billiard
in the plane. We used there the ideas from our recent papers on ordinary
Birkhoff billiards \cite{BM1}, \cite{BM2} extending previous
results of \cite{bolotin} and \cite{Tab}.

In the present paper we continue even further and examine algebraic integrability of magnetic billiards on the surfaces of constant curvature. As one can guess the result in this case interpolates planar magnetic case and ordinary billiard on the constant curvature case. Interestingly, our approach combines differential geometry on constant curvature surfaces with the algebraic geometry of curves. Algebraic integrability of ordinary Birkhoff billiards on constant curvature surfaces were studied in \cite{bolotin2} and recently in \cite{BM2} and \cite{g2}. It is very plausible that using the ideas of \cite{g1}, \cite{g2} one can complete the algebraic version of magnetic Birkhoff conjecture for the plane and constant curvature surfaces.

In this paper we are concerned with the existence of first integrals polynomial in
the velocities for magnetic billiards. The polynomial integrals are defined as follows:
\begin{definition}\label{def} Let $\Phi:T_1\Omega\rightarrow\mathbf{R}$
	be a function on the unit tangent bundle which is a polynomial in the components of the unit tangent vector $v$ with respect to a coordinate system $(\phi,\psi)$ on $\Sigma$
	$$\Phi(x,v)=\sum_{k+l=0}^N
	a_{kl}(\phi,\psi)v_{\phi}^kv_{\psi}^l$$
	with coefficients
	continuous up to the boundary, $a_{kl}\in C(\overline\Omega).$ We
	call $\Phi$ a polynomial integral of the magnetic billiard if the
	following conditions hold.
	
	1. $\Phi$ is an integral of the magnetic flow $g^t$ inside $\Omega$,
	$$\Phi(g^t(x,v))=\Phi(x,v);$$
	
	2. $\Phi$ is preserved under the reflections at the
	boundary $\partial\Omega$: for any
	$x\in\partial\Omega,$
	$$\Phi(x,v)=\Phi(x,v-2\langle n,v\rangle n),$$ for any $v\in T_x\Omega, |v|=1,$
	where $n$ is the unit normal to $\partial\Omega$ at $x$.
\end{definition}
Notice that the definition of polynomial integral does not depend on the choice of coordinates on $\Sigma$.
Using algebraic-geometry tools we shall prove the following:
\begin{theorem}\label{all}
	For any non-circular domain $\Omega$ on $\Sigma$, the magnetic
	billiard inside $\Omega$ is not algebraically integrable for all but
	finitely many values of $\beta$.
\end{theorem}
Moreover we shall show below in Theorem \ref{main} that for the existence of Polynomial integral of motion, the parallel curves $\gamma_{\pm r}$ of the boundary $\gamma=\partial \Omega$ must be
non-singular algebraic curves in $\mathbf C^3$.

In what follows, we realize $\Sigma$ in $\mathbf{R}^3$ as the
standard unit sphere with the induced metric from $\mathbf{R}^3$, for the case of $K=1$, and as the upper sheet of the hyperboloid
$\{x_1^2+x_2^2-x_3^2=-1\}$ endowed with the metric
$ds^2=dx_1^2+dx_2^2-dx_3^2$, for the case $K=-1.$
It is convenient to introduce the diagonal matrices for the spherical and hyperbolic case respectively:
$$A=\rm{diag}\{1,1,1\} , \quad\quad A=\rm{diag}\{1,1,-1\}. $$
Then the upper sheet of the hyperboloid gets the form: $$\{<Ax,x>=-1\}\subset\mathbf{R}^3,\quad ds^2=<Adx,dx>$$ endowed with Lorentzian metric. In what follows we fix the orientation on $\Sigma$ by the unite normal which at the point $x$ equals $x$. The corresponding complex structure on $\Sigma$ will be denoted by $J$.

\section{\bf Parallel curves}
Let $\gamma$ be an oriented (the domain $\Omega$ bounded by $\gamma$ lies to the left) simple closed curve on $\Sigma$ parametrized by the arc-length $s$.
In what follows the central role is played by the curves $\gamma_{+t}, \gamma_{-t}$ on $\Sigma$ defined for any given $t>0$ by the formulas:
\begin{equation}\label{e1}
\gamma_{+t}=\exp (tJ\gamma'(s)),\quad \gamma_{-t}=\exp (-tJ\gamma'(s)),
\end{equation}
where $J$ is the complex structure (rotation by $\pi /2$ in the tangent plane with respect to the orientation determined by the normal to $\Sigma$; the normal at the point $x$ equals $x$) and $\exp$ is the exponential map of $\Sigma$.
Here and below, we write $\gamma'$ and $\dot Y$ for the derivatives with respect to $s$ and $t$ respectively.

These curves $\gamma_{\pm t}$ have many names. They are called parallel curves, equidistant curves, fronts or offset curves.  Let $\rho$ denotes geodesic curvature radius of $\gamma$. Then, $k=\cot \rho$ for $K=+1$, as for $K=-1$ for $\gamma$  which is convex with respect to horocycles, we have $k=\coth \rho$.

The following \textit {perestroika} occurs for any non-circular convex curve $\gamma$ on $\Sigma$:
\begin{proposition}\label{perestroika}
\begin{enumerate}[(A)]\
	\item  If $0<t<\rho_{min}$  then parallel curves $ \gamma_{+t}$ are smooth convex curves.
	\item If  $\rho_{min}\leq t\leq \rho_{max}$ then $ \gamma_{+t}$ necessarily has singularities.
	\item  If $\rho_{max}< t<\pi/2$ for the case $K=+1$ and $\rho_{max}\leq t$ for $K=-1$ then $ \gamma_{+t}$ is smooth again.
	\item The curve $ \gamma_{-t}$ is smooth and convex for any $t\in(0, \pi/2)$ in the case $K=1$ and for any positive $t$ for $K=-1$.
\end{enumerate}
\end{proposition}

\begin{figure}[h]
	\centering
	\includegraphics[width=0.4\linewidth]{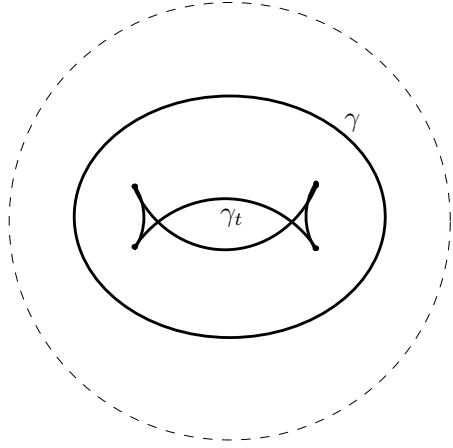}
	\caption{Singularities of the front $\gamma_t$}
	\label{2}
\end{figure}

\begin{proof}
Consider the family of geodesics $\gamma_s(t)=\exp (tJ\gamma'(s))$. Zeros of the Jacobi field $Y(s,t):=\partial_s \gamma_s(t) $ corresponding to this family is responsible for singularities of parallel curves. We have
$$\ddot Y+KY=0,\quad Y(s,0)=1,\quad \dot Y(s,0)=-k(s).$$
So we have
$$Y(s,t)=\cos t-k(s) \sin t,\quad for\quad K=1,$$
$$Y(s,t)=\cosh t-k (s)\sinh t,\quad for\quad K=-1.
$$
Then $$Y=0\Leftrightarrow  \cot t= k(s),$$	for the case $K=+1$
and$$
Y=0\Leftrightarrow  \coth t= k(s),
$$for the case $K=-1$.
This fact implies all the cases of Proposition \ref{perestroika}.
\end{proof}
Moreover one can easily derive the formulas for the geodesic curvature of the parallel curves for $t=\pm r$:
\begin{proposition}\label{curvature}
	\begin{enumerate}[(A)]\
\item $K=+1\Rightarrow$ $$ k_{+r}=\frac{\cot\rho \cot r+1}{\cot\rho -\cot r}=\frac{k\beta+1}{k-\beta},\ k_{-r}=\frac{\cot\rho \cot r-1}{\cot\rho +\cot r}=\frac{k\beta-1}{k+\beta}.
$$
\item $K=-1\Rightarrow$
$$ k_{+r}=\frac{\coth\rho \coth r-1}{\coth\rho -\coth r}=\frac{k\beta-1}{k-\beta},\ k_{-r}=\frac{\coth\rho \coth r+1}{\coth r+\coth\rho }=\frac{k\beta+1}{k+\beta}.
$$
\end{enumerate}
\end{proposition}
\begin{proof}
It is enough to prove the formulas for circles, because at any point $\gamma$ can be approximated by the osculating circle. 	For the circles the formulas follow immediately from the trigonometry formulas of addition  for the functions $\cot, \coth$.
	\end{proof}
The following inequalities are immediate and will be crucial:
\begin{proposition}\label{inequality}
	\begin{enumerate}[(A)]\
	
		\item $K=1,\  k>\beta\quad\Rightarrow\quad
		 k_{+r}>\beta,\quad
		k_{-r}< \beta,$
		
		\item $K=-1,\  k>\beta>1\quad\Rightarrow \quad
		 k_{+r}>\beta,\quad
		1<k_{-r}< \beta.$
		
		In particular, in both Spherical and Hyperbolic cases $$k_{\pm r}\neq \beta.$$
	\end{enumerate}
\end{proposition}

\section{\bf Larmor circles; the phase space of the magnetic billiard on $\Sigma$}\label{Larmor}

Recall that for the constant magnetic field of magnitude  $\beta$, the trajectories of the magnetic flows are
geodesic circles of radius $r$.

Throughout this paper we shall use the following construction.
Denote by $J$ the standard complex structure on $\Sigma$ and
introduce the mapping
\begin{equation}\label{L}{\mathcal L}: T_1\Omega \rightarrow
\Sigma,\quad
{\mathcal L}(x,v)=\exp_x(rJv),
\end{equation}
which assigns to every unit tangent vector $v\in T_x{\Omega}$ the
center of the unique Larmor circle passing through $x$ in the direction of $v$. Varying the unit vector
$v$ in $T_x{\Omega}$, for a fixed point $x\in\Omega$,  the
corresponding Larmor centers form a geodesic circle of radius $r$ centered at
$x$. The domain swept by all these circles when $x$ runs over
$\Omega$ will be denoted by $\Omega_r.$ Vice versa,  for any circle of radius $r$ lying in $\Omega_r$ its center
necessarily belongs to $\Omega$.

The domain $\Omega_r\subset\Sigma$ is a bounded domain
homeomorphic to an annulus and the curves $\gamma_{\pm r}$, are exactly the boundaries of $\Omega_r.$  Here $\gamma_{-r}$ lies on the outer boundary
of the annulus, and $\gamma_{+r}$ lies on the inner boundary.
Moreover it follows from Proposition \ref{curvature} that any circle of radius $r$ with the
center at $\gamma(s)$ is tangent to the outer boundary from inside
at $\exp_{\gamma(s)}(-rJ{\gamma}'(s))$, and to the inner boundary from
outside at the point $\exp_{\gamma(s)}(rJ{\gamma}'(s)).$ Moreover, apart
from these tangencies, this circle remains entirely inside
$\Omega_r$ (see Fig. \ref{1}).

\begin{figure}[h]
	\centering
	\includegraphics[width=0.4\linewidth]{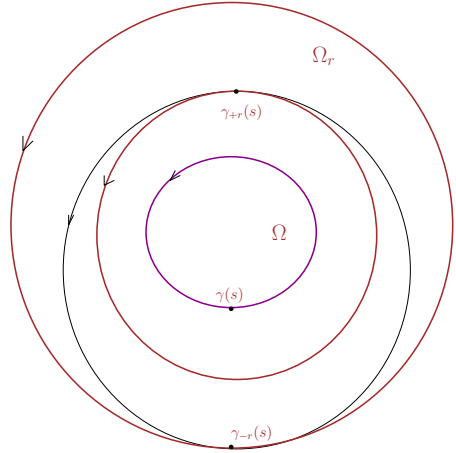}
	\caption{Circle of radius $r$ centered at $\gamma(s)$ is tangent to the boundary curves at $\gamma_{+r}(s)$ and $\gamma_{-r}(s)$  }
	\label{1}
\end{figure}

 In the sequel we need the formulas calculating the Larmor centers in the Spherical and Hyperbolic geometries.
 For the case of the standard unite sphere in $\mathbf {R}^3$ we choose positive normal at $x$ to be $x$ and have for the Larmor center:
\begin{equation}\label{e2}
\mathcal L (x,v)=\frac{\beta}{\sqrt{\beta^2+1}}x+\frac{1}{\sqrt{\beta^2+1}}[x,v]=\cos r\cdot x+\sin r [x,v],
\end{equation} where $[x,v]$ is just a positive unit normal vector to $v$  on $\Sigma$.

In the case of hyperboloid
we choose again  the positive normal to $\Sigma$ equal to $x$ and positive normal vector to $v$ on $\Sigma$ equals in this case $A[x,v] $.
So in the Hyperbolic case we have the formula:
\begin{equation}\label{e3}
\mathcal L (x,v)=\frac{\beta}{\sqrt{\beta^2-1}}x+                               \frac{1}{\sqrt{\beta^2-1}}A[x,v]=
\cosh r\cdot x+\sinh r A[x,v].
\end{equation}

Moreover, we introduce the mapping $$\M:\Omega_r\rightarrow\Omega_r$$
by the following rule: Let $C_-$ and $C_+$ be two Larmor circles
centered at $P_-$ and $ P_+$, respectively. We define
$$
\M(P_-)=P_+ \iff\ C_-\quad  \textrm{is\ transformed \ to}\quad C_+
$$
after billiard reflection at the boundary $\partial\Omega.$

It then follows easily that $\M:\Omega_r\rightarrow\Omega_r$ preserves the
standard symplectic form (area form) of $\Sigma$ and thus $\Omega_r$ naturally
becomes the phase space of the magnetic Birkhoff billiard. We shall
call $\M$ the magnetic billiard map.

Given a polynomial integral $\Phi$
of the magnetic billiard, we define the
function $F:\overline{\Omega}_r\rightarrow\overline{\Omega}_r$ by the requirement
\begin{equation}\label{F}
F\circ{\mathcal L}=\Phi.
\end{equation}
This is a well-defined construction, since $\Phi$ is an
integral of the magnetic flow, and therefore takes constant values
on any Larmor circle. Moreover, since $\Phi$ is invariant under the
billiard flow, $F$ is invariant under the billiard map $\M$:
\begin{equation}\label{invariance}
F\circ\M=F.
\end{equation}
Notice that since $\Phi$ is a polynomial in $v$ of degree $N$, the function $F$
satisfies the following  property: $F$ restricted to any circle of geodesic
radius $r$ lying in $\Omega_r$ is a trigonometric polynomial of
degree at most $N$.
Indeed, the Larmor circle centered at $x$ is obtained when the unit tangent vector $v$ varies in $T_x\Sigma$.
Choosing local coordinates $(\phi,\psi)$ which are Euclidean at the point $x\in\Sigma$ we have
$$\Phi=\sum_{ k+l=0}^N a_{kl}(x)v_{\phi}^kv_{\psi}^l , \quad v_{\phi}=\cos t, v_{\psi}=\sin t,$$
so $F$ indeed becomes a trigonometric polynomial in $t$.
The next theorem claims that in such a case $F$
is a restriction to $\Sigma$ of a polynomial function in $\mathbf R^3$. This theorem holds true both for Spherical and Hyperbolic case.
\begin{theorem}\label{harmonics} Let $\Omega_r$ be a domain in
	$\Sigma$ which is the union of all circles of radius $r$ whose
	centers run over a domain $\Omega$. Let $F:\overline{\Omega}_r\rightarrow\mathbf{R}$ be a
	continuous function such that the restriction of $F$ to any circle
	of radius $r$ of $\Omega_r$ is a trigonometric polynomial of degree
	at most $N.$ Then $F$ coincides with a restriction to $\Sigma$ of a polynomial function $\hat{F}$  in $(x_1,x_2,x_3)$ of degree at most
	$N.$
\end{theorem}
We shall prove this theorem below in Section \ref{two}.

Moreover, we will prove the following consequence of Theorem \ref{harmonics} which enables one to apply algebraic geometry methods:
\begin{proposition}\label{prop}
	Suppose that the magnetic billiard in $\Omega$ admits a polynomial
	integral $\Phi$ and let $\hat{F}$ be the corresponding polynomial. Then
	$$\hat{F}|_{\gamma_{\pm r}}={\rm const}.$$
\end{proposition}
\begin{corollary}\label{circle}
	If magnetic billiard in $\Omega$ has integral $\Phi$ which is linear in velocities, then $\Omega$ is a round disc on $\Sigma$.
\end{corollary}
\begin{proof}
	Indeed, by Theorem \ref{harmonics} $\hat F$ also has degree 1 and therefore by Proposition \ref{prop} the closed curves $\gamma_{\pm r}$ are intersections of $\Sigma$ with 2-planes, thus circles. Hence also $\gamma$ is a circle.
	\end{proof}

In view of the Corollary \ref{circle} we shall assume everywhere below that the degree $N>1$.

\begin{remark}\label{r}
	One can assume that
	the polynomial  $\hat{F}$ is such that the ${\rm
		constant}$ in Proposition \ref{prop} is $0,$ for both parallel
	curves $\gamma_{\pm r}$. Indeed, if $\hat{F}|_{\gamma_{- r}}=c_1$ and
	$\hat{F}|_{\gamma_{+ r}}=c_2,$ one can replace $\hat{F}$ by
	$\hat{F}^2-(c_1+c_2)\hat{F}+c_1\cdot c_2$ to annihilate both constants
	$c_1,c_2.$
\end{remark}

Since the curves $\gamma_{\pm r}$ lie in $\{\hat F=0\}\cap \{\Lambda-1=0\}$ we have:

\begin{corollary}\label{algebraic}
	 The curves ${\gamma}_{\pm r}$ and hence also ${\gamma}$ determine irreducible algebraic sets denoted by  $\hat{\gamma}_{\pm r}$ and $\hat{\gamma}$ in $\mathbf{C}^3$.
\end{corollary}		
	
Next we use the folkloric fact that a curve of bidegree $(m,m)$ on a quadric is a complete intersection, i.e., intersection of the quadric with a surface of degree $m$ (we refer to \cite{H} for more details). Moreover,
using an appropriate  version of Noether theorem (\cite{W}, p. 226, Chapter VIII) we can summarize the needed algebraic-geometry facts:
\begin{theorem}\label{algebr}\
\begin{enumerate}[A.]

\item
The ideal of irreducible algebraic sets $\hat{\gamma}_{\pm r}$	is generated by two polynomials $F_{\pm r}$ and $(\Lambda - 1)$,
where $F_{\pm r}$ is irreducible in the ring \newline $\mathbf C[x_1,x_2,x_3]/{\rm mod} (\Lambda -1)$.
\item
In addition we have:
\begin{equation}
\label{B}
\hat F= F_{\pm r}^k\cdot g_{\pm}\  {\rm mod}(\Lambda -1),
\end{equation}
where polynomials $g_{\pm}$ do not vanish on $\hat{\gamma}_{\pm r}$ but in finitely many points.
\item
At all but finitely many points of $\hat{\gamma}_{\pm r}$ the differentials $DF_{\pm r}$ and $D\Lambda$  are not proportional, which means that the differential of the function  $F_{\pm r}|_{\Sigma}$ does not vanish on $\hat{\gamma}_{\pm r}$.
\end{enumerate}
\end{theorem}
In the sequel we shall use the following homogenization.
Given a polynomial function on $\Sigma$ we extend it to the homogeneous function in the space away from the cone $\{\Lambda=0\}$. The extended homogeneous function we shall write with tilde.
Thus for  $\hat F=\sum_{k= 0}^N \hat F_k$ which is a sum of homogeneous components $\hat F_k$ of degree $k$, we
define
\begin{equation}\label{tilde}
\tilde F= \sum_{k= 0}^N \hat F_k \sqrt {\Lambda}^{N-k},
\end{equation}
where as above $\Lambda=(x_1^2+x_2^2+x_3^2)$ for the sphere, and $\Lambda=(-x_1^2-x_2^2+x_3^2)$ for the hyperboloid.
Then obviously, $\tilde F$ is a homogeneous function of degree $N$ and $$\tilde F|_{\Sigma}=\hat F|_{\Sigma}=F.$$
Similarly we define homogeneous functions $\tilde F_{\pm r}$ and $\tilde g _{\pm}$. So equation (\ref{B}) in homogeneous form reads:
\begin{equation}\label{k}
\tilde F=\tilde F_{\pm r}^k\cdot \tilde g_{\pm}	
\end{equation}
Moreover from the very construction the functions
$$
\tilde F;\tilde F_{\pm r}; \tilde g_{\pm}	
$$
all have the form
\begin{equation}\label{form}
p+q\sqrt{\Lambda},
\end{equation}
where $p$, $q$ are some homogeneous polynomials with the degree of $q$ is by one less than that of $p$. Therefore we have an important:

\begin{remark}
Function $\tilde F$  is analytic away from the absolute $\{\Lambda=0\}$.
\end{remark}

\subsection{Main result and example}
We now turn to the formulation of our main result:
\begin{theorem}\label{main}
	Let $\Omega$ be a convex bounded domain on $\Sigma$ with smooth
	boundary $\gamma=\partial\Omega$ which has curvature at least $\beta$ ($\beta>1$ in the Hyperbolic case). Suppose that the
	magnetic billiard in $\Omega$ admits a polynomial integral $\Phi.$
	Then the curves $\hat \gamma_{\pm r}$ are smooth algebraic curves in $\mathbf{C}^3$.
\end{theorem}

Having this result it is easy to give a proof of Theorem \ref{all} .

\begin{proof}[Proof of Theorem \ref{all}]
	Indeed, it is easy to check that the polynomials $ {F}_{\pm r}$ depend on the variable $d=\tan r$ (in the Hyperbolic case $d=\tanh r)$ in a polynomial way, so
	$ {F}_{\pm r}$ is a polynomial in $x_1,x_2,x_3$, and $d$.
	Moreover, since
	$\gamma$ has positive curvature bounded from below by $\beta$, there
	is a whole open interval $r\in(\rho_{\min},\rho_{\max})$ where by Proposition \ref{perestroika} the parallel curve
	$\gamma_{+r}$ does have real singularities.
	Hence, the system of
	equations
	$$
	\partial_{x _1}\tilde F_{+r}=	\partial_{x _2}\tilde F_{+r}=	\partial_{x _3}\tilde F_{+r}=\tilde F_{+r}=\Lambda-1=0
	$$
	defines an algebraic set in ${\mathbf C}^4$ and its projection on
	the
	$d$-coordinate line is a Zariski open set.
	It then follows that singularities persist for all
	but finitely many $d.$
\end{proof}

\noindent {\bf Example.} Let $\Omega$ be the interior of the ellipse on the sphere, i.e. the intersection of the sphere with a quadratic cone
$$\partial\Omega=\left\{\frac{x_1^2}{a^2}+\frac{x_2^2}{b^2}=x_3^2\right\},\quad 0<b<a.$$
The equation of parallel curves for the ellipse is defined by the polynomial $\hat{F}$ of degree eight (see Appendix). The curve $\{\hat{F}=0\}$ on the sphere  is singular for arbitrary $a$ and $b$. For $a=2, b=1$ we have
$$
 \hat{F}=
((d^2-4)^2-10 (4+5 d^2+d^4) x_1^2+25 (1+d^2)^2 x_1^4) (5 x_1^2+d^2 (3+5 x_1^2)-3)^2+
$$
$$
4(1+d^2) (5 (1+d^2)(124+70 d^2+31 d^4) x_1^2-3 (32+60 d^2-45 d^4+7 d^6)-375 (1+d^2)^2\times
$$
$$
(3+d^2) x_1^4+625 (1+d^2)^3 x_1^6) x_2^2+4 (1+d^2)^2
(73 d^4-248 d^2-32-150 (4+7 d^2+3 d^4) x_1^2+
$$
$$
825
(1+d^2)^2 x_1^4) x_2^4+
64 (1+d^2)^3 (8-7 d^2+25(1+d^2) x_1^2) x_2^6+256(1+d^2)^4 x_2^8.
$$
By direct calculation we checked that the curve $\{\hat{F}=0\}$ is irreducible in the ring $\mathbf C[x_1,x_2,x_3]/{\rm mod} (\Lambda -1)$. Hence, by Theorem \ref{main} the magnetic billiard inside the ellipse is algebraically non-integrable for any magnitude of the magnetic field. It is plausible that the curve $\hat{F}=0$ is irreducible for arbitrary $a$ and $b$.

\section{\bf Proof of the main theorem}
The main step in the proof of main Theorem \ref{main} is the following result.
We stick to notations of Section \ref{Larmor}.

\begin{theorem}\label {constant} Let $\tilde F$ be a homogeneous function of degree $N$ which coincides with $F$ on $\Sigma$. Assume $\nabla \tilde F$ does not vanish on $\gamma_{\pm r}$ except for finitely many points. Then the identity
	\begin{equation}\label{both}
	\frac{\rm Hess\  \tilde F}{(N-1)^2}\mp \beta |\nabla\tilde F|^3=\const
	\end{equation}
	holds true for all points on the curve $\gamma_{\pm r}$.
	
	Moreover the constant  in the RHS of (\ref{both}) is not zero.
\end{theorem}
We shall prove Theorem \ref{constant} in Section \ref{cube}.
Now we are in position to complete the proof of the main Theorem \ref{main}.

\begin{proof}[Proof of Theorem \ref{main}]
In order to fulfill the assumption of Theorem \ref{constant} we need to pass
from  $\tilde F$ to $\tilde F^{\frac{1}{k}}$ and use equation (\ref{k}).
  Notice that  $\tilde F^{\frac{1}{k}}$ is also conserved by the map $\M$.
	In the proof we consider  the curve $\gamma_{+
		r}$ (the proof for the curve $\gamma_{-r}$ is identical).
	It then follows from Theorem \ref{algebr} that we can write
	$$ \tilde F=\tilde F_{+r}^k\cdot \tilde g_+,
	$$ for some integer $k\geq 1$, so that $\nabla \tilde F_{+r}$ and $\tilde g_+$ do not
	vanish on  $\gamma_{+ r}$ except for finitely many points. Let us take some arc $\delta$ of $\gamma_{+ r}$ with this property.
	We may assume that $\tilde g_+$ is positive on $\delta$ (otherwise we change the signs).
	Therefore the equation (\ref{both}) can be
	derived in the same manner for the function
	$$\tilde F^{\frac{1}{k}}=\tilde F_{+r}\cdot \tilde g_+^{\frac{1}{k}},$$
	for all points of the arc $\delta$.
	This function is homogeneous of degree
	$$
	p=N/k>1.
	$$ Thus we have instead of (\ref{both}):
	\begin{equation}\label{rem4}
	\frac{1}{(p-1)^2}{\rm Hess} (\tilde F_{+r}\cdot \tilde g_+^{\frac{1}{k}})+\beta|\nabla (\tilde F_{+r}\cdot
	\tilde g_+^{\frac{1}{k}})|^3={\rm const}, \quad (x_1,x_2,x_3)\in\delta.
	\end{equation}
	Moreover the constant on the RHS of (\ref{rem4}) is not zero, by Theorem \ref{constant}.
	Using the identities
	$$
	{\rm Hess}(\tilde F_{+r}\cdot \tilde g_+^{\frac{1}{k}})=\tilde g_+^{\frac{3}{k}}{\rm Hess}(\tilde F_{+r}),\quad \nabla
	(\tilde F_{+r}\cdot \tilde g_+^{\frac{1}{k}})=\tilde g_+^{\frac{1}{k}}\nabla (\tilde F_{+r}),
	$$ which are valid for all $(x_1,x_2,x_3)\in\{\tilde F_{+r}=0\}$,
	we obtain from (\ref{rem4}) that
	\begin{equation}\label{rem44}
	\tilde g_+^{\frac{3}{k}}\left({\frac{1}{(p-1)^2}\rm Hess}(\tilde F_{+r})+\beta|\nabla \tilde F_{+r}|^3\right)={\rm const}, \quad
	(x_1,x_2,x_3)\in\delta.
	\end{equation}
	Raising back to the power $k$ we get
	\begin{equation}\label{rem5}
	\tilde g_+^{{3}}\left({\frac{1}{(p-1)^2}\rm Hess}(\tilde F_{+r})+\beta|\nabla \tilde F_{+r}|^3\right)^k={\rm const}, \
	(x_1,x_2,x_3)\in\delta.
	\end{equation}
	Let us prove now, that the curve $\hat\gamma_{+r}$ is non-singular in $\mathbf C^3$.
	We argue by contradiction. Suppose, there exist a singular point $P$ of $\hat\gamma_{+r}\subset \{\Lambda-1=0\} \subset\mathbf C^3$.
	Pick any point $Q\in\delta$, and consider a path $\alpha$ on the algebraic curve $\hat\gamma_{+r}$ going from $Q$ to $P$ avoiding
	the singular points of $\tilde F_{+r}$.
	Using the particular form of $\tilde F_{+r}$ and  $\tilde g_+$ given by Theorem \ref{algebraic} we see that the equation
	(\ref{rem5}) remains valid for
	analytic continuation of the functions $\tilde F_{+r}$,  $\tilde g_+$   along the path $\alpha$. Hence it remains valid also at the point $P$. But this is impossible, because the LHS is zero at $P$ while the constant is not zero at the RHS.
	The proof is completed.
\end{proof}

 \section{\bf Boundary values of the integral.}
\label{three} In this Section we prove Proposition \ref{prop}.

  Take a point $x$ on $\gamma$.  Let $v$ be a positive unit
 tangent vector to $\gamma$. Let $C_-$ and $C_+$ be the incoming and
 outgoing circles with the unit tangent vectors $v_-$ and $v_+$ at
 the impact point $x$. We are interested in the two cases when the
 reflection angle between $v$ and $v_-$ is close to $0$ or to $\pi$. These two possibilities correspond to the following cases:

 \begin{equation}\label{ab}
 (a)\quad\quad v_-=R_{-\epsilon}v,\quad\ \quad v_+=R_{\epsilon}v,
 \end{equation}
 $$(b)\quad v_-=R_{\epsilon}(-v),\quad v_+=R_{-\epsilon}(-v),$$
 where $R_{\epsilon}$ is the counterclockwise rotation of the tangent plane $T_x\Sigma$
 by a small angle $\epsilon$, see Fig. \ref{3}.

 We define
 $$ P_-(\epsilon)={\mathcal L}(x,v_-)=\exp_x(rJ(v_-)),
 \quad P_+(\epsilon)={\mathcal L}(x,v_+)=\exp_x(rJ(v_+)).
 $$
 In the case (a) we have
 \begin{equation}\label{a}
 P_-(\epsilon)=\exp_x (rJ(R_{-\epsilon}v)),\quad
 P_+(\epsilon)=\exp_x (rJ(R_{\epsilon}v)).
 \end{equation}
 In the case (b), \begin{equation}\label{b}
 P_-(\epsilon)=\exp_x (-rJ(R_{\epsilon}v)),\quad
 P_+(\epsilon)=\exp_x (-rJ(R_{-\epsilon}v)).
 \end{equation}
We write:
$$P_{-}:=P_-(\epsilon),\ P_+:=P_+(\epsilon),P_0:=P_-(0)=P_+(0).$$
\begin{figure}[h]
	\centering
  \includegraphics[width=0.5\linewidth]{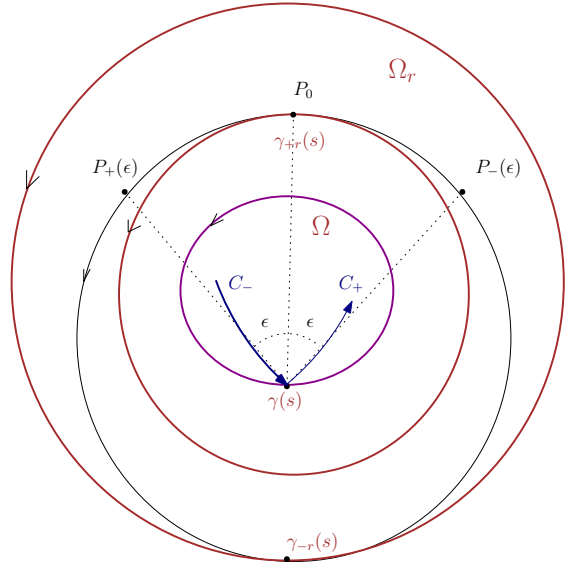}
	\caption{The center $P_-$ of the circle $C_-$ is mapped by $\M$ to the center $P_+$ of $C_+$ }
	\label{3}
\end{figure}

 Notice that in case (a) the middle point of the short arc that
 connects the points $P_-(\epsilon)$ and $P_+(\epsilon)$ is
 $P_0=\exp_x(rJ(v))\in\gamma_{+r}$, while for the case (b) the middle
 point is $P_0=\exp_x(rJ(-v))\in\gamma_{-r}$.
 \begin{proof}[Proof of Proposition \ref{prop}]
 	The condition 2. of Definition \ref{def} reads in terms of $F$
 	\begin{equation}\label{P}
 	F(P_-(\epsilon))=F(P_+(\epsilon)).
 	\end{equation}
 	Differentiating this equality with respect to $\epsilon$ at
 	$\epsilon=0$
 	we compute in
 	the case (a):
 	$$\frac{d}{d\epsilon}{|_{_{\epsilon=0}}}F(P_-(\epsilon))=dF|_{P_0}\left (\frac{d}{d\epsilon}{|_{_{\epsilon=0}}}P_-(\epsilon)\right )=dF|_{P_0}(Y(r)),$$
 	$$\frac{d}{d\epsilon}|_{\epsilon=0}F(P_+(\epsilon))=dF|_{P_0}\left (\frac{d}{d\epsilon}{|_{_{\epsilon=0}}}P_+(\epsilon)\right )=dF|_{P_0}(-Y(r)),
 	$$
 where $Y(t)$ is an orthogonal Jacobi field along the geodesic $\exp_x(tJv)$ corresponding to the radial family of geodesics $\exp_x(tR_{\epsilon}Jv)$. Thus $Y(r)$ is tangent to $\gamma_{+r}$ at the point $P_0$.
 The last two equalities together imply :
 $$
 dF|_{P_0}(Y(r))=0,
 $$ so $F$ has a constant value on $\gamma_{+r}$.
 Analogously one treats the case (b).
This completes the proof of Proposition \ref{prop}
\end{proof}

\section{\bf Remarkable equation}

In this Section we deduce the  remarkable equation expressing (\ref{invariance}) for a function $F:\overline{\Omega}_r\rightarrow\overline{\Omega}_r$ invariant under the map $\mathcal M$. This equation is valid for those $F$
 which have non-vanishing gradient
at a point on the boundary of $\Omega_r$. Moreover we may assume that $n_{\pm r}=A\frac{\nabla F}{|\nabla F|}$ is a positive unite normal to $\gamma_{\pm r}$, i.e. the basis $(\dot{\gamma}_{\pm r}, n_{\pm r})$ is positive (otherwise we change the sign of $F$).
In order to perform computations we need to rewrite the equation (\ref{invariance}) and hence (\ref{P})  in terms of the homogeneous function $\tilde F$ defined by (\ref{tilde}).

 We fix a point $(x_0,\pm v_0), \ v_0=\dot \gamma,\ n_0=Jv_0$ and rewrite (\ref{P}):
\begin{equation}\label{PP}
F({\mathcal L}(x_0,R_{-\epsilon}(\pm v_0)))=F({\mathcal L}(x_0,R_{+\epsilon}(\pm v_0)))
,
\end{equation}
where the sign $+$ and $-$ correspond to the cases  (a) and  (b) in (\ref{ab}) respectively.

In what follows we develop (\ref{ee}) in $\epsilon $ at $\epsilon=0$. This will contain derivatives $\tilde F$
at the point $P_0={\mathcal L}(x_0,v_0)$ for the case (a) (at $P_0={\mathcal L}(x_0,-v_0)$ for the case (b)).
We shall denote by $y_0$ the point $P_0$ viewed in $\mathbf{R}^3$.
Notice that since $\gamma_{\pm r}$ are parallel curves to $\gamma$ their tangent vector at $P_0$ viewed in $\mathbf{R}^3$ is exactly $\mp v_0$.
Moreover it follows from Euler formula for $\tilde F$ that the positive unite normal to $\gamma_{\pm r}$ at $P_0$ as a vector of $\mathbf{R}^3$ equals
$$
n_{\pm r}(y_0)=A\frac{\nabla\tilde F}{|\nabla\tilde F|}\left(y_0\right).
$$
In addition,
\begin{equation}\label{v0}
v_0=-A\frac{[y_0,A\nabla\tilde F]}{|\nabla\tilde F|}\left(y_0\right).
\end{equation}
One more thing we need is to rewrite positive normal vector $n_0$ to $\gamma$ at $x_0$ via $\tilde F$ at the point $y_0$. This can be done as follows.  Vectors $n_0$ and $n_{\pm r}$ are unite vectors which are tangent (with opposite orientation) to the same geodesic $\exp_{x_0}(tJv_0)$ at the points which are at distance $r$ apart.  Then we compute in the spherical case:
\begin{equation}\label{n0s}
n_0=\pm(-\cos r\cdot n_{\pm r}(y_0)+\sin r \cdot y_0)=\pm(-\cos r\cdot \frac{\nabla\tilde F}{|\nabla\tilde F|}\left(y_0\right)+\sin r \cdot y_0).
\end{equation}
In the hyperbolic case this formula reads:
\begin{equation}\label{n0h}
n_0=\pm(-\cosh r\cdot n_{\pm r}(y_0)-\sinh r \cdot y_0)=\pm(-\cosh r\cdot A \frac{\nabla\tilde F}{|\nabla\tilde F|}\left(y_0\right)-\sinh r \cdot y_0).
\end{equation}
Now we shall substitute into (\ref{PP})  the expression
$$
R_{\epsilon}(\pm v_0)=\cos \epsilon (\pm v_0)+\sin \epsilon (\pm n_0)
$$
together with the explicit formulas (\ref{v0}), (\ref{n0s}), (\ref{n0h}).

We shall consider the cases of sphere and hyperboloid separately.

Case 1.  In the case of sphere we get by the substitution the RHS of (\ref{PP}) :
\begin{equation}\label{ee}F({\mathcal L}(x_0,R_{+\epsilon}(\pm v_0)))=
\tilde F\left (\cos r x_0+\sin r[x_0,R_{\epsilon}(\pm v_0)]\right)=
\end{equation}
$$=\tilde F\left (\cos r x_0+\sin r[x_0,(\cos \epsilon (\pm v_0)+\sin \epsilon (\pm n_0))]\right)=$$
$$
=\tilde F\left (y_0+\sin r\(\cos {\epsilon}-1)[x_0,(\pm v_0)]+\sin r\sin \epsilon [x_0,(\pm n_0)]\right)=
$$
$$
=\tilde F\left (y_0\pm\sin r\(\cos {\epsilon}-1)n_0\mp\sin r\sin \epsilon v_0\right)=
$$
$$
=\tilde F\left (y_0+\sin r\(\cos {\epsilon}-1)\left(-\cos r\cdot \frac{\nabla\tilde F}{|\nabla\tilde F|}\left(y_0\right)+\sin r \cdot y_0\right)\pm\right.
$$
$$
 \left.\sin r\sin \epsilon \frac{[y_0,\nabla\tilde F]}{|\nabla\tilde F|}\left(y_0\right)\right).
$$
So finally we can rewrite equation (\ref{PP}) as
\begin{equation}\label{PPP}
\tilde F\left (y_0+\sin r\(\cos {\epsilon}-1)\left(-\cos r\cdot \frac{\nabla\tilde F}{|\nabla\tilde F|}\left(y_0\right)+\sin r \cdot y_0\right)\mp\right.
\end{equation}
$$
 \left.\sin r\sin \epsilon \frac{[y_0,\nabla\tilde F]}{|\nabla\tilde F|}\left(y_0\right)\right)
$$
$$
=\tilde F\left (y_0+\sin r\(\cos {\epsilon}-1)\left(-\cos r\cdot \frac{\nabla\tilde F}{|\nabla\tilde F|}\left(y_0\right)+\sin r \cdot y_0\right)\pm\right.
$$
$$
 \left.\sin r\sin \epsilon \frac{[y_0,\nabla\tilde F]}{|\nabla\tilde F|}\left(y_0\right)\right).
$$
Case 2. In the hyperbolic case formulas are similar:

\begin{equation}\label{eeh}F({\mathcal L}(x_0,R_{+\epsilon}(\pm v_0)))=
\tilde F\left (\cosh r x_0+\sinh rA[x_0,R_{\epsilon}(\pm v_0)]\right)=
\end{equation}
$$=\tilde F\left (\cosh r x_0+\sinh rA[x_0,(\cos \epsilon (\pm v_0)+\sin \epsilon (\pm n_0))]\right)=$$
$$
=\tilde F\left (y_0+\sinh r\(\cos {\epsilon}-1)A[x_0,(\pm v_0)]+\sinh r\sin \epsilon A [x_0,(\pm n_0)]\right)=
$$
$$
=\tilde F\left (y_0\pm\sinh r\(\cos {\epsilon}-1)n_0\mp\sinh r\sin \epsilon v_0\right)=
$$
$$
=\tilde F\left (y_0-\sinh r\(\cos {\epsilon}-1)\left(\sinh r \cdot y_0+\cosh r\cdot A \frac{\nabla\tilde F}{|\nabla\tilde F|}\left(y_0\right)\right)\pm\right.
$$
$$
 \left.\sinh r\sin \epsilon A\frac{[y_0,A\nabla\tilde F]}{|\nabla\tilde F|}\left(y_0\right)\right).
$$
Thus in the hyperbolic case we get finally:
\begin{equation}\label{PPPh}
\tilde F\left (y_0-\sinh r\(\cos {\epsilon}-1)\left(\sinh r \cdot y_0+\cosh r\cdot A \frac{\nabla\tilde F}{|\nabla\tilde F|}\left(y_0\right)\right)\mp\right.
\end{equation}
$$
 \left.\sinh r\sin \epsilon A\frac{[y_0,A \nabla\tilde F]}{|\nabla\tilde F|}\left(y_0\right)\right)=
$$
$$
=\tilde F\left (y_0-\sinh r\(\cos {\epsilon}-1)\left(\sinh r \cdot y_0+\cosh r\cdot A \frac{\nabla\tilde F}{|\nabla\tilde F|}\left(y_0\right)\right)\pm\right.
$$
$$
 \left.\sinh r\sin \epsilon A\frac{[y_0,A \nabla\tilde F]}{|\nabla\tilde F|}\left(y_0\right)\right).
$$
\section{\bf Terms of $\epsilon^3$ and proof of Theorem \ref{constant}.}\label{cube}
In this Section we compute and put in a very compact form terms of order $\epsilon^3$ of equations (\ref{PPP}), (\ref{PPPh}). Then we prove Theorem \ref{constant}.

In order to write the terms of order $\epsilon ^3$ of the equations (\ref{PPP}), (\ref{PPPh}) we first rewrite the argument of (\ref{PPP}):
\begin{equation}\label{PPP1}
 y_0+\sin r\(\cos {\epsilon}-1)\left(-\cos r\cdot \frac{\nabla\tilde F}{|\nabla\tilde F|}\left(y_0\right)+\sin r \cdot y_0\right)\mp\sin r\sin \epsilon \frac{[y_0,\nabla\tilde F]}{|\nabla\tilde F|}\left(y_0\right)=
\end{equation}
$$
 \left.=y_0(1-2\sin^2r\sin^2 (\epsilon/2))+2\sin r\cos r\sin^2(\epsilon/2)\frac{\nabla\tilde F}{|\nabla\tilde F|}(y_0)\mp\right.
$$
$$
 \sin r\sin \epsilon \frac{[y_0,\nabla\tilde F]}{|\nabla\tilde F|}\left(y_0\right)=
$$
$$
(1-2\sin^2r\sin^2 (\epsilon/2))\left(y_0+
\frac{\nabla\tilde F}{|\nabla\tilde F|}\frac{2\sin r\cos r\sin^2(\epsilon/2)}{1-2\sin^2 r\sin^2 (\epsilon/2) }\mp\right.
$$
$$
 \left.\frac{\sin r\sin \epsilon }{1-2\sin^2 r\sin^2 (\epsilon/2)}\frac{[y_0,\nabla\tilde F]}{|\nabla\tilde F|}\right).
$$
We can neglect the factor in front of the brackets, since the function $\tilde F$ is homogeneous. So up to order $\epsilon ^3$ we have for the argument :

$$
y_0+\frac{\sin r\cos r }{2}\epsilon^2\frac{\nabla\tilde F}{|\nabla\tilde F|}\mp \left(\sin r \epsilon+\left(\frac{1}{2}\sin^3 r-\frac{1}{6}\sin r\right )\epsilon^3\right)\cdot\frac{[y_0,\nabla\tilde F]}{|\nabla\tilde F|}.
$$
Using this one can write the third order expansion for (\ref{PPP}) and analogously for (\ref{PPPh}).  It turns out that the terms of order $\epsilon ^3$ in the equations (\ref{PPP}), (\ref{PPPh}) can be organized so that they are complete derivatives along the tangent vector $v$ to the curves $\gamma_{\pm r}$ (it is given through $F$ by formula(\ref{v0})). On the Sphere this reads:
\begin{equation}\label {Lv}
L_v\left(\frac{\rm Hess\  \tilde F}{(N-1)^2}\mp \cot r |\nabla\tilde F|^3\right)=0.
\end{equation}
As for the Hyperbolic case:
\begin{equation}\label {Lvh}
L_v\left(\frac{\rm Hess\  \tilde F}{(N-1)^2}\mp \coth r |\nabla\tilde F|^3\right)=0.
\end{equation}
So in both cases we can write (\ref{Lv}) and (\ref{Lvh}) as
\begin{equation}\label{lvboth}
L_v\left(\frac{\rm Hess\  \tilde F}{(N-1)^2}\mp \beta |\nabla\tilde F|^3\right)=0.
\end{equation}
Now we are in position to prove Theorem \ref{constant}.

\begin{proof}[Proof of Theorem \ref{constant}]
	The identity (\ref{both}) follows from (\ref{lvboth}). We need to show that the constant in the RHS of (\ref{lvboth}) is not zero.
	For the proof we need the following important formula for the geodesic curvature of curves on $\Sigma$ which we prove in Section \ref{homogen}.
\begin{lemma}\label{hom}
	Let $\tilde F$ be a homogeneous function in $\mathbf{R}^3$ of degree $N>1$.
		The geodesic curvature of the curve $\{\tilde F=0\}$ on $\Sigma$ (with respect to the normal $A\frac{\nabla\tilde F}{|\nabla\tilde F|}$) at a non-singular point
		is given in both geometries by the same formula:
		\begin{equation}\label{kg}
		k=\frac{\rm Hess\  \tilde F}{(N-1)^2|\nabla\tilde F|^3}.
		\end{equation}
	\end{lemma}
We prove the Lemma in Section \ref{homogen}.

It follows from Lemma \ref{hom} that if the constant in (\ref{both}) is zero then the geodesic curvature of the curve $\gamma_{\pm r}$ equals $\pm \beta$. But this cannot happen due to the inequality of Proposition \ref{inequality}.
This completes the proof of Theorem \ref{constant}.
\end{proof}

\section {\bf Geodesic curvature in homogeneous coordinates.}\label{homogen}

	\begin{proof}[Proof of Lemma \ref{hom}]
		 Set positive normal and the tangent vector to the curve at the point $x$:
		 \begin{equation}\label{nv}
		 n=A\frac{\nabla\tilde F}{|\nabla\tilde F|},\quad v=-A\frac{[x,A\nabla\tilde F]}{|\nabla\tilde F|}.\end{equation}
		Then we use the formula for geodesic curvature on $\Sigma$ using Frenet formula:
		$$
		k=-<\nabla_v n, Av>,
		$$
		where $\nabla$ is the standard flat connection on $\mathbf{R}^3$.
		$$
		k=-\left<\nabla_v\left(A\frac{\nabla\tilde F}{|\nabla\tilde F|}
\right) , Av\right>=-\frac{1}{|\nabla\tilde F|}<\nabla_v\nabla \tilde F,v>=-\frac{1}{|\nabla\tilde F|}<Hv,v>,
		$$
		where $H$ denotes the  matrix of second derivatives of $\tilde F$. Notice that we omitted another term containing the derivative of $\frac{1}{|\nabla\tilde F|}$ using the fact that $<A\nabla \tilde F, Av>=0$. Next we use Euler identities for the derivatives of $\tilde F$:
		\begin{equation}\label{euler}
Hx=(N-1)\nabla \tilde F.
		\end{equation}
		Thus we continue using (\ref{nv}), (\ref{euler}):
		$$
		k=-\frac{1}{|\nabla\tilde F|^3}<H[Ax,\nabla \tilde F], [Ax,\nabla \tilde F]>=
		$$
		 $$
		 =-\frac{1}{(N-1)^2|\nabla\tilde F|^3}<H[Ax,Hx], [Ax,Hx]>=
		 $$
		$$
		 =-\frac{{\rm Hess }\tilde F}{(N-1)^2|\nabla\tilde F|^3}<[Ax,Hx], [H^{-1}Ax,x]>=
		$$
		$$
		=-\frac{{\rm Hess }\tilde F}{(N-1)^2|\nabla\tilde F|^3}\left(<Ax,H^{-1}Ax><Hx,x>-<Ax,x><Hx, H^{-1}Ax>\right),
		$$
		where ${\rm Hess }\tilde F$ is the determinant of $H$.
		
		At last we notice that
		$$
		<Hx,x>=(N-1)<\nabla\tilde F,x>=(N-1)N\cdot\tilde F(x)=0
		,$$ by Euler identity. Thus we get
		$$
		k=\frac{{\rm Hess }\tilde F}{(N-1)^2|\nabla\tilde F|^3}<Ax,x>^2=\frac{{\rm Hess }\tilde F}{(N-1)^2|\nabla\tilde F|^3},
		$$
		which yields (\ref {kg}).
	\end{proof}

 \section{\bf Proof of Theorem \ref{harmonics}}\label{two}
 In the proof we follow our strategy from \cite{BM3}.

 \begin{proof} Let us assume first that $F$ is a $C^\infty$-function. We shall say that $F$ has
 	property $P_N$ if the restriction of $F$ to any circle of radius $r$
 	lying in $\Omega_r$ is a trigonometric polynomial of degree at most $N$.
 	The proof of Theorem \ref{harmonics} goes by induction on the degree $N$.
 	
 	1) For $N=0$, the lemma obviously holds, since if $F$ has property
 	$P_0$, then $F$ is a constant on any circle of radius $r$ and hence
 	must be a constant on the whole $\Omega_r$, because any two points of $\Omega_r$
 	can be connected by a union of a finite number of circular arcs of
 	radius $r$.
 	
 	2) Assume now that any function satisfying property $P_{N-1}$ is a restriction to $\Sigma$ of a polynomial function in $(x_1,x_2,x_3)$ of degree at most
 	$N-1.$
 	
 	Let $F$ be any smooth function on $\Omega_r$ with property $P_N$. Fix a point $x_*\in\Omega$ and consider  the circle
 	$C_0$ of radius $r$ centered in $x_*$. Applying an appropriate isometry of $\Sigma$ we may assume that $$x_*=(0,0,1)\in\mathbf{R}^3 \quad  {\rm and} \quad  C_0=\{x_3=h\},$$
 	where $h=\cos r$ in the spherical case and $h=\cosh r$ for hyperboloid.
 	Obviously there exists a polynomial $\hat{F}_0$ in $(x_1,x_2,x_3)$ of degree at most $N$ such that    satisfying
 	${F}|_{C_0}=\hat F_0|_{C_0}.$ Then applying Hadamard's lemma to the function $F-\hat F_0$
 	one can find a $C^{\infty}$ function
 	$G:\Omega_r\rightarrow\mathbf{R}$ such that
 	\begin{equation}\label{F0}
 	F(x)-\hat{F_0}(x)=(x_3-h)G(x), \quad \forall x \in \Omega_r.
 	\end{equation}
 	
 Let us show now
 	that $G$ has property $P_{N-1}$. Then by induction we will have that
 	$G$ is a restriction to $\Sigma$ of a polynomial $\hat{G}$ of degree $(N-1)$ and thus by (\ref{F0}), $F$
 	is a polynomial of degree at most $N$. Thus we need to show that the
 	function $g:=G|_C$ is a trigonometric polynomial of degree at most
 	$N-1$, for any circle $C$ of radius $r$ lying in $\Omega_r$.
 We can apply suitable rotation along the $x_3$-axes so that the center of $C$ has $x_1=0$. Let us denote by $\alpha$ the geodesic distance between the centers of $C_0$ and $C$ in $\Omega$.
 	
 	We shall split the proof in two cases.
 	
 	Case 1. If $\Sigma$ is the sphere, we can write $C$ as follows
 	$$
 	C=R_{\alpha} C_0, \quad R_{\alpha}=
 	\begin{pmatrix}
 	1&0&0 \\ 0&\cos\alpha&-\sin\alpha\\0&\sin\alpha&\cos\alpha
 	\end{pmatrix}.
 	$$
 	Therefore parametrizing $C_0$ we compute the parametrization of $C$ as follows
 	$$
 	C_0=(\sin r\sin t,\sin r\cos t,\cos r)\Rightarrow
 	C=(*,  *,\  \sin r\cos t\sin \alpha +\cos\alpha\cos r).
 	$$
 	Then we compute $(x_3-h)|_{C}$:
 	\begin{equation}\label{x3}
 	(x_3-h)|_{C}=\sin r\cos t\sin \alpha +\cos\alpha\cos r-\cos r=\end{equation}
 	$$=\sin r\sin\alpha\left(\frac{e^{it}+e^{-it}}{2}\right)-\cos r(1-\cos\alpha).
 	$$

 	
 	Substituting (\ref{x3}) into (\ref{F0}) we get
 	$$
 	(F-\hat F_0)|_C=(\sin r\sin\alpha\left(\frac{e^{it}+e^{-it}}{2}\right)-\cos r(1-\cos\alpha))\cdot G|_C.
 	$$
 	Expanding the left- and the right-hand sides in Fourier series we
 	get
 	$$
 	\sum_{-\infty}^{+\infty}f_ke^{ikt}=(
 	\sin r\sin\alpha\left(\frac{e^{it}+e^{-it}}{2}\right)-\cos r(1-\cos\alpha))\sum_{-\infty}^{+\infty}g_ke^{ikt},
 	$$ where $f_k$ are Fourier coefficients of $(F-\hat F_0)|_C$.
 	Moreover, we have $$f_k=0,\quad  |k|>N,$$ since both $F$ and $\hat F_0$
 	have property $P_N.$ Thus we obtain a linear recurrence relation for
 	the coefficients $g_k$:
 	$$
 	\sin\alpha\sin r\cdot g_{k+1}-4\cos r\sin^2(\alpha/2)\cdot g_k+\sin\alpha\sin r \cdot g_{k-1}=0, \quad  |k|>N.
 	$$
 	The characteristic polynomial of this difference equation,
 	$$\lambda^2-2\cot r\tan(\alpha/2)\lambda+1=0,$$ has the discriminant $$D=\cot^2 r\tan^2(\alpha/2)-1$$ which is strictly negative
 	due to the inequality
 	$\alpha/2<r$, which holds true for the following reason:  the inequality of our setup, $\cot r=\beta<k$ implies that the distance between any two points of $\Omega$ is less than $2r$, in particular $\alpha<2r$.
 	Therefore
 	the characteristic equation has two complex conjugate roots
 	$\lambda_{1,2}=e^{\pm i\phi}$
 	and therefore we can write
 	$$
 	g_{N+l}=c_1e^{il\phi}+c_2e^{-il\phi}, \quad  l\geq2,
 	$$ where
 	$$c_1+c_2=g_N,\quad c_1e^{i\phi}+c_2e^{-i\phi}=g_{N+1}.
 	$$
 	It is obvious now that if at least one of the coefficients $g_N$ or
 	$g_{N+1}$ does not vanish, then at least one of the constants $c_1,
 	c_2$ does not vanish, and therefore the sequence $\{g_{N+l}\}$ does
 	not converge to $0$ when $l\rightarrow +\infty$. This contradicts
 	the continuity of $g$. Therefore, both $g_N$ and $g_{N+1}$ must
 	vanish, and so $g$ is a trigonometric polynomial of degree at most
 	$N-1$, proving that $G$ has property $P_{N-1}$. This completes the
 	proof of Theorem \ref{harmonics} for $C^{\infty}$-case for the sphere.
 	
 	Case 2. If $\Sigma$ is the upper sheet of the hyperboloid, we can write $C$ as follows
 	$$
 	C=Rh_{\alpha} C_0, \quad Rh_{\alpha}=
 	\begin{pmatrix}
 	1&0&0 \\ 0&\cosh\alpha&\sinh\alpha\\0&\sinh\alpha&\cosh\alpha
 	\end{pmatrix}.
 	$$
 	Therefore parametrizing $C_0$ we compute the parametrization of $C$ as follows
 	$$
 	C_0=(\sinh r\sin t,\sinh r\cos t,\cosh r)\Rightarrow
 	C=(*,  *,\  \sinh r\cos t\sinh \alpha +\cosh\alpha\cosh r).
 	$$
 	Then we compute $(x_3-h)|_{C}$:
 	\begin{equation}\label{x3h}
 	(x_3-h)|_{C}=\sinh r\cos t\sinh \alpha +\cosh\alpha\cosh r-\cosh r=\end{equation}
 	$$=\sinh r\sinh\alpha\left(\frac{e^{it}+e^{-it}}{2}\right)-\cosh r(1-\cosh\alpha).
 	$$
 	Substituting (\ref{x3h}) into (\ref{F0}) we get
 	$$
 	(F-\hat F_0)|_C=(\sinh r\sinh\alpha\left(\frac{e^{it}+e^{-it}}{2}\right)-\cosh r(1-\cosh\alpha))\cdot G|_C.
 	$$
 	Expanding again the left- and the right-hand sides in Fourier series we
 	get
 	$$
 	\sum_{-\infty}^{+\infty}f_ke^{ikt}=(
 	\sinh r\sinh\alpha\left(\frac{e^{it}+e^{-it}}{2}\right)-\cosh r(1-\cosh\alpha))\sum_{-\infty}^{+\infty}g_ke^{ikt},
 	$$ where $f_k$ are Fourier coefficients of $(F-\hat F_0)|_C$.
 	Moreover, we have $$f_k=0,\quad  |k|>N,$$ since both $F$ and $\hat F_0$
 	have property $P_N.$ Thus the linear recurrence relation in Case 2 reads:
 	$$
 	\sinh\alpha\sinh r\cdot g_{k+1}-4\cosh r\sinh^2(\alpha/2)\cdot g_k+\sinh\alpha\sinh r \cdot g_{k-1}=0, \quad  |k|>N.
 	$$
 	The characteristic polynomial of this difference equation,
 	$$\lambda^2-2\coth r\cdot\tanh(\alpha/2)\lambda+1=0,$$ has the discriminant $$D=\coth^2r\cdot\tanh^2(\alpha/2)-1$$ which is  again strictly negative
 	due to the inequality
 	$\alpha/2<r$ as in the previous case.
 	Therefore
 	the characteristic equation has two complex conjugate roots
 	and we finish exactly as in the Case 1. This completes the
 	proof of Theorem \ref{harmonics} for $C^{\infty}$-case for the hyperboloid.

 	The general case when $F$ is only continuous, can be proven by a limiting argument
 	exactly as we did in \cite{BM3}. We omit the details.
 \end{proof}

\section{Appendix}
Let $\Omega$ be the interior of the ellipse on the sphere, i.e. the intersection of the sphere  with a quadratic cone
$$
 \partial\Omega=\left\{\frac{x_1^2}{a^2}+\frac{x_2^2}{b^2}=x_3^2\right\},\quad 0<b<a.
$$
The equation of the parallel curves for an ellipse reads $\hat{F}=0$, where
$$\hat{F}=(a^4 (x_2^2+d^2 (x_2^2-1))^2 ((1+a^2)^2 (1+d^2)^2 x_1^4+2 (1+a^2) (1+d^2)
x_1^2 \times
$$
$$
(x_2^2-a^2+d^2(x_2^2-1))+(a^2+x_2^2+d^2(x_2^2-1))^2)+2a^2b^2(a^6 ((1+d^2)^2 x_2^4-d^2-
$$
$$
(1+d^2)^2 x_2^2)-(1+a^2)^2 (1+d^2)^2 x_1^4 (3
d^2 (a^2+d^2)+(1+d^2) (a^2 (3+d^2)-d^2) x_2^2-
$$
$$
(2+a^2) (1+d^2)^2 x_2^4)+
a^4 (d^4+d^2
(2+5 d^2+3 d^4) x_2^2-3 (1+d^2)^2 (1+2 d^2) x_2^4+
$$
$$
3 (1+d^2)^3 x_2^6)+(1+a^2) (1+d^2) x_1^2 (3 a^4 d^2+2 a^2 d^4+3 d^6+(1+d^2)(3 a^2 d^2(d^2-1)-
$$
$$
5 d^4+
a^4
(3+2 d^2)) x_2^2-(1+d^2)^2 (2 a^4-d^2+a^2 (6 d^2-1)) x_2^4+(1+3 a^2) (1+d^2)^3
x_2^6)+
$$
$$
 a^2(d^2-2 (1+d^2)^2 x_2^2+2 (1+d^2)^2 x_2^4) (x_2^2+d^2 (x_2^2-1))^2+d^2(x_2^2+d^2
(x_2^2-1))^3+
$$
$$
 (1+a^2)^3(1+d^2)^3 x_1^6 (x_2^2+d^2(1+x_2^2)))+b^8 ((1+d^2) x_2^2-1)^2 ((1+a^2)^2 (1+d^2)^2 x_1^4+
$$
$$
2 (1+a^2) (1+d^2) x_1^2
(a^2 ((1+d^2) x_2^2-1)-d^2)+(d^2+a^2 ((1+d^2) x_2^2-1))^2)+
$$
$$
b^4 ((1+a^2)^4(1+d^2)^4 x_1^8+a^8 ((1+d^2) x_2^2-1)^2+ 2(1+a^2)^3 (1+d^2)^3
 x_1^6((1+a^2)\times
$$
$$
(1+d^2) x_2^2-2 (a^2+d^2))+
 (1+a^2)^2(1+d^2)^2 x_1^4(6 a^4+10 a^2
d^2+6 d^4+2 (1+d^2)\times
$$
$$
 (a^2 (1+d^2)-3 a^4-3 d^2) x_2^2+ (1+8 a^2+a^4) (1+d^2)^2 x_2^4)+
2 a^6 (d^2+(3+5 d^2+2 d^4) x_2^2-
$$
$$
 3 (1+d^2)^2 (2+d^2)x_2^4+3(1+d^2)^3 x_2^6)+2
 a^2 d^2(d^4+d^2
(2+5 d^2+3 d^4) x_2^2-3 (1+d^2)^2 \times
$$
$$
 (1+2 d^2) x_2^4+3(1+d^2)^3 x_2^6)+2 a^4 ((1+d^2)^2 (3+10 d^2+3 d^4) x_2^4-3 d^4-4 (d+d^3)^2
x_2^2-
$$
$$
6 (1+d^2)^4 x_2^6+3 (1+d^2)^4 x_2^8)+d^4 (x_2^2+d^2 (x_2^2-1))^2-
2 (1+a^2) (1+d^2) x_1^2 (a^6 (2-
$$
$$
 3 (1+d^2) x_2^2+(1+d^2)^2
 x_2^4)+d^2 (2 d^4-3 d^2 (1+d^2) x_2^2+(1+d^2)^2 x_2^4)+a^2 (4 d^4+d^2 (3+
$$
$$
8 d^2+5 d^4) x_2^2-(1+d^2)^2
(2 d^2-3) x_2^4-3 (1+d^2)^3 x_2^6)+a^4(4 d^2+(5+8 d^2+3 d^4) x_2^2+
$$
$$
(1+d^2)^2 (3 d^2-2)
x_2^4-3 (1+d^2)^3 x_2^6)))+
2 b^6(a^6 ((1+d^2)x_1^2-1+(1+d^2) x_2^2)^2\times
$$
$$
((1+d^2)x_2^2-1+(1+d^2) x_1^2 (1+(1+d^2)
x_2^2))+(x_1^2+d^2(x_1^2-1))^2 ((1+d^2)^2 x_2^4-d^2-
$$
$$
(1+d^2)^2 x_2^2+(1+d^2) x_1^2
(1+(1+d^2) x_2^2))+a^2 (3(1+d^2)^3 x_1^6(1+(1+d^2) x_2^2)+
$$
$$
 (1+d^2)^2 x_1^4(4(1+d^2)^2 x_2^4-(1+7 d^2+6 d^4)
 x_2^2-3-6 d^2)+d^2(d^2+(3+5 d^2+2 d^4) x_2^2-
$$
$$
 3(1+d^2)^2(2+d^2) x_2^4+3(1+d^2)^3
 x_2^6)+(1+d^2) x_1^2((3+2 d^2+2 d^4+3 d^6)x_2^2+
$$
$$
d^2(2+3d^2)-(1+d^2)^2(6+d^2) x_2^4+3
(1+d^2)^3 x_2^6))+a^4(3(1+d^2)^3 x_1^6(1+(1+d^2) x_2^2)+
$$
$$
 ((1+d^2) x_2^2-1)^2(d^2-2(1+d^2)^2
x_2^2+2(1+d^2)^2 x_2^4)+(1+d^2)^2 x_1^4((1-2 d^2-3 d^4)x_2^2$$
$$
-3(2+d^2)+5(1+d^2)^2
x_2^4)+(1+d^2) x_1^2(3+2 d^2-(2+5 d^2+3 d^4) x_2^2+
$$
$$
(d^2-5)(1+d^2)^2 x_2^4+4(1+d^2)^3
x_2 ^6)))).
$$
The curve on the sphere defined by the equation $\hat{F}=0$ is singular. For example, it has the following
singular points
$$
 \left(\pm\frac{\sqrt{(a^2-b^2)(b^2-d^2)}}{b\sqrt{(1+a^2)(1+d^2)}},0,\pm\sqrt{\frac{2+8d^2}{5+5d^2}}
 \right).
$$

\section*{\bf Acknowledgements}We are grateful to
Eugene Shustin for indispensable  consultations on algebraic geometry questions.

\end{document}